\documentclass{amsart}

\title[Relatively extra-large Artin groups]{The $K(\pi,1)$ conjecture and acylindrical hyperbolicity for relatively extra-large Artin groups}
\author{Katherine Goldman}
\address{231 West 18th Avenue, Columbus OH 43210}
\email{goldman.224@osu.edu}
\urladdr{goldmanmath.com}

\usepackage{amsmath,amsthm,amssymb}
\usepackage[hidelinks]{hyperref}
\usepackage{mathtools}

\usepackage{cite}

\usepackage{tikz}
\usetikzlibrary{decorations.markings}

\usepackage{tikz-cd}

\usepackage{thmtools}
\usepackage{thm-restate}

\usepackage{microtype}

\newtheorem*{thm*}{Theorem}
\newtheorem{thm}{Theorem}[section]
\newtheorem{prop}[thm]{Proposition}
\newtheorem{lemma}[thm]{Lemma}

\theoremstyle{definition} 
\newtheorem{defn}[thm]{Definition}
\newtheorem{rem}[thm]{Remark}

\begin{document}

\begin{abstract}
    Let $A_\Gamma$ be an Artin group with defining graph $\Gamma$. We introduce the notion of $A_\Gamma$ being extra-large relative to a family of arbitrary parabolic subgroups. This generalizes a related notion of $A_\Gamma$ being extra-large relative to two parabolic subgroups, one of which is always large type. Under this new condition, we show that $A_\Gamma$ satisfies the $K(\pi,1)$ conjecture whenever each of the distinguished subgroups do. In addition, we show that $A_\Gamma$ is acylindrically hyperbolic under only mild conditions. 
\end{abstract}

\maketitle

Let $\Gamma$ be a finite simplicial graph whose edges are labeled with (finite) integers, each at least $2$. For vertices $s,t$ of $\Gamma$ connected by an edge, let $m(s,t)$ denote the label of the edge between $s$ and $t$. Let $S = \mathrm{Vert}(\Gamma)$. Since $\Gamma$ is simplicial, we use the convention that an edge of $\Gamma$ is the same as an unordered pair $\{s,t\}$ of vertices of $\Gamma$. The Artin group defined by $\Gamma$ is
\[
    A_\Gamma = \langle\, S \mid \mathrm{prod}(s,t; m(s,t)) = \mathrm{prod}(t,s; m(s,t)) \text{ for } \{s,t\} \text{ an edge of } \Gamma \,\rangle,
\]
where $\mathrm{prod}(a,b;n)$ is the alternating word in $a$ and $b$ starting with $a$ of length $n$ (eg, $aba \dots$). We call the pair $(A,S)$ an \emph{Artin-Tits system}.

There is a Coxeter group also naturally associated with this defining graph; namely,
\[
    W_\Gamma = \langle\, S \mid (st)^{m(s,t)} = 1 \text{ for } \{s,t\} \text{ an edge of } \Gamma, \ s^2 = 1 \text{ for } s \in S \,\rangle.
\]
It is well known that there is a natural surjective homomorphism $A_\Gamma \to W_\Gamma$ induced by the identity map on $S$. Recall that if $W_\Gamma$ is finite, then we call $W_\Gamma$ \emph{spherical} and call $A_\Gamma$ \emph{spherical-type}. In this case, we may sometimes refer to $\Gamma$ itself as \emph{spherical-type}. 

By van der Lek \cite{vdl1983}, if $\Gamma'$ is a full (or ``induced'') subgraph of $\Gamma$, then the natural map from the Artin group $A_{\Gamma'}$ to $A_\Gamma$ is an injection. (Recall that a subgraph $\Gamma'$ of $\Gamma$ is called full if for any pair of vertices $v,w$ of $\Gamma'$ which span an edge $\{v,w\}$ in $\Gamma$, we also have that $\{v,w\}$ is an edge of $\Gamma'$.)  
We call such a subgroup of $A_\Gamma$ a \emph{(standard) parabolic subgroup}. Sometimes, if $T = \mathrm{Vert}(\Gamma')$, we write $A_T$ for $A_{\Gamma'}$.

It is also well known that the Artin group $A_\Gamma$ is the fundamental group of a space $N(W)$ which is the quotient of a complement of a certain complexified hyperplane arrangement by a natural $W_\Gamma$-action. (See \cite{paris2014k} for more details.)
The long-standing $K(\pi,1)$ conjecture states that $N(W)$ is aspherical (ie, has contractible universal cover).
Currently, the $K(\pi,1)$ conjecture is known to be true when
\begin{enumerate}
    \item $A_\Gamma$ is spherical-type (in \cite{deligne1972immeubles}),
    \item $A_\Gamma$ is affine-type (meaning $W_\Gamma$ has a finite-index subgroup which acts properly by isometries on a Euclidean space) (proven in general in \cite{paolini2021proof}), 
    \item if $\Gamma'$ is a full spherical-type subgraph of $\Gamma$ then $|\mathrm{Vert}(\Gamma')| = 2$ (in which case $A_\Gamma$ is called 2-dimensional) (in \cite{charney1995k}), or more generally, if $A_\Gamma$ is locally reducible (in \cite{Charney2000TheTC}),
    \item if every full complete subgraph of $\Gamma$ is spherical-type (in which case $A_\Gamma$ is called FC type) 
    (also in \cite{charney1995k}), and
    \item some combination criteria are satisfied, including results by Godelle and Paris \cite{godelle2012k} and Ellis and Sk\"oldberg \cite{ellis2010}.
\end{enumerate}

We present a new criterion based on the following familiar condition: an Artin group $A_\Gamma$ is \emph{extra-large type} if every edge of $\Gamma$ has label at least 4.
In this case, $A_\Gamma$ is 2-dimensional, and thus satisfies the $K(\pi,1)$ conjecture.
In \cite{juhasz2018relatively}, the following condition is introduced. 
Let $H = A_{\Gamma'}$ be a standard parabolic subgroup of $A$ (with $\Gamma' \subseteq \Gamma$ a full subgraph). Then $A$ is \emph{extra-large relative to $H$} (or \emph{$\Gamma'$-relatively extra-large}) if 
\begin{enumerate}
    \item for every edge $\{s,t\}$ of $\Gamma$ with $s \in \Gamma'$ and $t \not\in \Gamma'$, we have $m(s,t) \geq 4$, and
    \item for every edge $\{t,t'\}$ of $\Gamma$ with $t,t' \not\in \Gamma'$, we have $m(t,t') \geq 3$.
\end{enumerate}
It is then shown that $A_\Gamma$ satisfies the word problem or $K(\pi,1)$ conjecture whenever $H$ does. It is in this spirit that we make the following generalization.

Let $\{\Gamma_i\}$ be a finite family consisting of disjoint, non-empty full subgraphs of $\Gamma$ with vertex sets $S = \mathrm{Vert}(\Gamma)$ and $S_i = \mathrm{Vert}(\Gamma_i)$. Suppose also that $S = \bigcup S_i$. In direct analogy to the relatively extra-large condition, we consider
\begin{enumerate}
    \item[(REL)] Every edge of $\Gamma$ between 
     $\Gamma_i$ and  
     $\Gamma_j$ for some $i \not= j$ has label at least $4$.
\end{enumerate}
If this condition is satisfied, we say that $A_\Gamma$ is \emph{$\{\Gamma_i\}$-relatively extra-large}.
We establish the following theorem regarding such Artin groups.

\begin{thm*}
    Suppose $A_\Gamma$ is $\{\Gamma_i\}$-relatively extra-large. Then $A_\Gamma$ satisfies the $K(\pi,1)$ conjecture if and only if each $A_{\Gamma_i}$ does.
\end{thm*}

In fact, a somewhat stronger fact can be established using our methods. 
Instead of (REL), consider

\begin{enumerate}
    \item[(REL$'$)] If $e$ is an edge of $\Gamma$ between $\Gamma_i$ and $\Gamma_j$ for some $i \not= j$ and $e$ shares a vertex with a distinct edge between $\Gamma_i$ and $\Gamma_k$ for some $i \not= k$, then $e$ has label at least $4$.
\end{enumerate}
Specifically, this allows edges which are isolated among those edges between the subgraphs in the family $\{\Gamma_i\}$ to have label 2 or 3.
We show

\begin{restatable}{theorem}{relkpone}
\label{thm:relkp1}
    Suppose $\Gamma$ and $\{\Gamma_i\}$ satisfy \textnormal{(REL$'$)}. Then $A_\Gamma$ satisfies the $K(\pi,1)$ conjecture if and only if each $A_{\Gamma_i}$ does.
\end{restatable}

In addition to this, we are able to show under mild hypotheses that Artin groups satisfying (REL$'$) are also acylindrically hyperbolic.
Acylindrical hyperbolicity is a property of interest for many groups, including Artin groups. Some of the classes for which acylindrical hyperbolicity is known for include
\begin{enumerate}
    \item right-angled Artin groups ($m(s,t) = 2$ for each edge of $\Gamma$) which are not cyclic or a direct product of non-trivial subgroups (in \cite{osin2016acylindrically}),
    \item spherical-type Artin groups (in \cite{calvez2017}),
    \item type FC Artin groups whose defining graph has diameter at least $3$ (in \cite{chatterji2016note}),
    \item extra-extra-large type Artin groups (meaning $m(s,t) \geq 5$ for each edge $\{s,t\}$ of the defining graph) of rank at least $3$ (in \cite{haettel2019xxl}),
    \item Artin groups $A_\Gamma$ such that $\Gamma$ is not a join of two subgraphs $\Gamma_1$, $\Gamma_2$ (in \cite{charney2019artin}), 
    \item affine-type Artin groups (in \cite{calvez2020euclidean}),
    \item 2-dimensional Artin groups of hyperbolic type (meaning the associated Coxeter group is hyperbolic) (in \cite{martin2019acylindrical}), and
    \item 2-dimensional Artin groups (in \cite{vaskou2021acylindrical}). 
\end{enumerate}
We show acylindrical hyperbolicity in our setting as well:
\begin{restatable}{theorem}{acynhyp}
\label{thm:acynhyp}
    Suppose $A_\Gamma$ and $\{\Gamma_i\}_{i=1}^n$, $n \geq 2$ satisfy \textnormal{(REL$'$)}. In addition, assume $|\mathrm{Vert}(\Gamma)| \geq 3$ and not all edges between the family $\{\Gamma_i\}$ have label 2. Then $A_\Gamma$ is acylindrically hyperbolic.
\end{restatable}

We note that the conditions in Theorem \ref{thm:relkp1} include the original relatively extra-large condition of Juh\'asz as a special case. Suppose $A_\Gamma$ is $\Gamma'$-relatively extra-large (in the sense of \cite{juhasz2018relatively}). Let $\Gamma''$ be the full subgraph on the vertices of $\Gamma$ which are not in $\Gamma'$. Then $A_\Gamma$ is $\{\Gamma', \Gamma''\}$-relatively extra-large in our sense.
The condition (2) in the definition of $\Gamma$-relatively extra-large is equivalent to requiring that $A_{\Gamma''}$ be large type (ie, all edge labels are at least 3). Then
$A_{\Gamma''}$ satisfies the $K(\pi,1)$ conjecture as $A_{\Gamma''}$ is 2-dimensional. Thus according to our result, $A_\Gamma$ satisfies the $K(\pi,1)$ conjecture if and only if $A_{\Gamma'}$ does.

Our Theorems include many new examples for which the $K(\pi,1)$ conjecture and/or acylindrical hyperbolicity was not previously known. As one example, consider two graphs $\Gamma_1$, $\Gamma_2$ of type $\widetilde C_3$ (see Figure \ref{fig:C3artingp}).
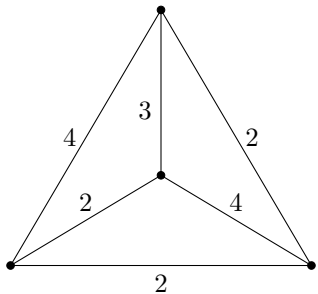
\begin{figure}[h!]
    \centering
    \def\scalepic{2}
    \begin{tikzpicture}
        \coordinate (a) at (0,0);
        \coordinate (b) at (0*\scalepic,1.1*\scalepic);
        \coordinate (c) at (1*\scalepic,-0.6*\scalepic);
        \coordinate (d) at (-1*\scalepic,-0.6*\scalepic);

        \draw (a) -- (b) node[midway, below left]  {3};
        \draw (a) -- (c) node[midway, above]  {4};
        \draw (a) -- (d) node[midway, above] {2};
        \draw (b) -- (c) node[midway, right]  {2};
        \draw (b) -- (d) node[midway, left] {4};
        \draw (d) -- (c) node[midway, below] {2};
        
        \draw[fill] (a) circle [radius=0.05];
        \draw[fill] (b) circle [radius=0.05];
        \draw[fill] (c) circle [radius=0.05];
        \draw[fill] (d) circle [radius=0.05];
    \end{tikzpicture}
    \caption{A defining graph of type $\widetilde C_3$} \label{fig:C3artingp}
\end{figure}
These defining graphs generate an affine Artin group, and thus satisfy the $K(\pi,1)$ conjecture.
Then, let $\Gamma$ be the join of $\Gamma_1$ and $\Gamma_2$ with each of the new edges labeled by $4$ (or greater). It is quickly checked that $A_\Gamma$ satisfies none of the previously listed conditions. But, $\Gamma$ and $\{\Gamma_1,\Gamma_2\}$ satisfy (REL$'$), and each $A_{\Gamma_i}$ satisfies the $K(\pi,1)$ conjecture, so $A_\Gamma$ does. In addition, none of the edges between $\Gamma_1$ and $\Gamma_2$ are labeled 2, so $A_\Gamma$ is acylindrically hyperbolic.
More generally, if a clique $\Gamma$ with at least 3 vertices is extra-large relative to a family $\{\Gamma_i\}$, then $A_\Gamma$ is acylindrically hyperbolic, and if each $\Gamma_i$ satisfies the $K(\pi,1)$ conjecture, then $A_\Gamma$ does as well.

We also note that our methodology for proving Theorem \ref{thm:relkp1} differs from Juh\'asz' original work, allowing us to drop his condition (2) and treat more general defining graphs. This also allows us to easily prove acylindrical hyperbolicity. We hope that this method may be adapted for other similar restrictions on $A_\Gamma$.
Namely, our strategy for proving Theorem \ref{thm:relkp1} is as follows. In Section 1, we construct a simplicial complex as a variation of the usual Deligne complex. We show the complex is $\mathrm{CAT}(0)$, hence contractible, in Section 2. Then in Section 3, we show that this complex is homotopy equivalent to the universal cover of $N(W)$ by a result of Godelle and Paris \cite{godelle2012k}. 
In Section 4, we prove Theorem \ref{thm:acynhyp} using recent results of Vaskou \cite{vaskou2021acylindrical}.

We would also like to note that the conditions (REL) and (REL$'$) can be naturally relaxed to allow edges with label at least 3, which would define a \emph{relatively large type} condition. This case is also currently of interest to the author; however, it is somewhat more complicated
than the current case of consideration. 

The author would like to extend great thanks to Mike Davis and Jingyin Huang for their helpful comments and advice given through the writing of this paper.

\section{The Deligne-like complex}

Before we define our complex, we wish to establish a lemma in Artin groups similar to a well-known property of cosets of standard parabolic subgroups of Coxeter groups. 
We include a proof for the reader's convenience. We make heavy use of this result in the subsequent sections. 

\begin{lemma} \label{lemma:cosetinclusion}
    Suppose $(A,S)$ is an Artin-Tits system, $\alpha,\alpha' \in A$, and $T, T' \subseteq S$. Then if $\alpha A_T \subseteq \alpha' A_{T'}$, we have $\alpha^{-1}\alpha' \in  A_{T'}$ and $T \subseteq T'$.
\end{lemma}

\begin{proof}
    Let $w$ and $w'$ be the image of $\alpha$ and $\alpha'$, respectively, under the quotient homomorphism $A_\Gamma \to W_\Gamma$. The inclusion $\alpha A_T \subseteq \alpha' A_{T'}$ is preserved under the quotient map, giving us the relation $w W_T \subseteq w' W_{T'}$ in $W_\Gamma$. So, by \cite[Ch.~IV \S 8 Th.~2(iii)]{bourbaki2008lie}, we have $T \subseteq T'$ as subsets of $W_\Gamma$. Since the quotient map is bijective on the generators, this gives $T \subseteq T'$ viewed in $A_\Gamma$.
    
    To see that $\alpha$ and $\alpha'$ must be in the same $A_{T'}$-coset, note that 
    $\alpha A_{T} \subseteq \alpha A_{T'}$ as well as $\alpha A_T \subseteq \alpha' A_{T'}$, so $\varnothing \not= \alpha A_T \subseteq \alpha A_{T'} \cap \alpha' A_{T'}$. Since cosets partition the group and these cosets have non-empty intersection, they must be the same.
\end{proof}

We also briefly give a restatement of a result of van der Lek.

\begin{lemma} \label{lemma:prodgen}
    If $(A,S)$ is an Artin-Tits system and $s \in S$, then $s$ cannot be written as a product of the elements of $S \setminus \{s\}$.
\end{lemma}

\begin{proof}
    By van der Lek's thesis \cite{vdl1983}, 
    \[
        A_{\{s\}} \cap A_{S \setminus \{s\}} \cong A_{\{s\} \cap S \setminus \{s\}} = A_{\varnothing} = 1.
    \]
    Thus in particular, $s \not \in A_{S \setminus \{s\}}$. Since $A_{S \setminus \{s\}}$ is the collection of all possible products of the generators $S \setminus \{s\}$, the result follows.
\end{proof}

\subsection{Definition of the complex}

Through the rest of the paper, we let $A = A_\Gamma$ be an Artin group such that $\Gamma$ and $\{\Gamma_i\}$ satisfy (REL$'$), with $S_i = \mathrm{Vert}(\Gamma_i)$ and $A_i = A_{\Gamma_i}$.

We now introduce a simplicial complex based on our distinguished subgroups $A_i$ of $A$ analogous to the Deligne complex. To do this, we mimic the construction of the Deligne complex in \cite{charney1995k}, but replace the poset of spherical generating sets with the following set.

\begin{defn}
    Let $\mathcal S^\ell$ be the set of all $T \subseteq S$ satisfying either
    \begin{enumerate}
        \item $T = \varnothing$ (in which case $A_T = 1$, the trivial subgroup of $A$),
        \item $T = S_i$,
        \item $T = \{s_i,s_j\}$ for vertices $s_i \in S_i$, $s_j \in S_j$ of an edge between $\Gamma_i$ and $\Gamma_j$ with $i\not=j$, or
        \item $T = \{s\}$ for a vertex $s$ of an edge between $\Gamma_i$ and $\Gamma_j$, $i\not= j$.
    \end{enumerate}
    With this, we define
    \[ A\mathcal S^\ell = \{\, \alpha A_T \, : \, \alpha \in A, T \in \mathcal S^\ell \,\}, \]
    and order these sets by inclusion.
    We then let $X$ denote the geometric realization of the derived complex of $\mathcal S^\ell$ and $\hat \Phi$ denote the geometric realization of the derived complex of $A \mathcal S^\ell$ (recall that the derived complex of a poset is the set of chains in the poset ordered by inclusion of chains).
\end{defn}

 We will denote an $n$-simplex of $\hat \Phi$ by
\[ [ \alpha_0 A_{T_0},\, \alpha_1 A_{T_1},\, \dots,\, \alpha_n A_{T_n} ]\]
where $\alpha_0 A_{T_0} < \alpha_1 A_{T_1} < \dots < \alpha_n A_{T_n}$ is a chain in $A \mathcal S^\ell$. We use similar notation for simplices of $X$.
Notice that $\hat \Phi$ inherits a natural left action of $A$ with fundamental domain isomorphic to $X$ via the simplicial map induced by the set map $T \mapsto A_T$. 
    
We note that if one replaces $\mathcal S^\ell$ by $\mathcal S^f$, the set of $T \subseteq S$ so that $A_T$ is spherical-type, then the definition of the (modified) Deligne complex of Charney and Davis \cite{charney1995k} is recovered. To further borrow their notation, we will let $K$ denote the geometric realization of the derived complex of $\mathcal S^f$, let $A \mathcal S^f$ denote the cosets of $A_T$ for $T \in \mathcal S^f$, and let $\Phi_M = \Phi_M(A_\Gamma)$ denote the geometric realization of the derived complex of $A \mathcal S^f$.

The rest of this section and the next is dedicated to showing that $\hat \Phi$ is $\mathrm{CAT}(0)$. First, we show that $\hat \Phi$ is simply connected, then endow it with a metric of non-positive curvature.

To show that $\hat \Phi$ is simply connected, we will use basic facts about complexes of groups. We will only need the fact that the action of $A_\Gamma$ on $\hat \Phi$ has a complex of groups structure briefly, so we will summarize the basic argument here, and refer the reader to \cite{Haefliger92Extension} for more details on complexes of groups. 

\begin{lemma} \label{lemma:simcon}
    The complex $\hat \Phi$ is simply connected.
\end{lemma}

\begin{proof}
    The stabilizer of a vertex $[\alpha A_T]$ of $\hat \Phi$ is the subgroup $\alpha A_T \alpha^{-1}$ of $A$. Thus $A$ acts on $\hat \Phi$ without inversion. The complex $X$ is homeomorphic to the quotient $\hat \Phi / A$ via the simplicial map induced by $T \mapsto A_T$. In addition, $X$ is simply connected, as $[\varnothing]$ is a cone point in $X$. This information determines a complex of groups \cite[Section 2.1]{Haefliger92Extension}, which we denote by $A(X)$. The edge maps are the usual inclusion maps $A_T \hookrightarrow A_{T'}$. Note that this complex is developable by definition. 
    
    Since $X$ is simply connected, $\pi_1(A(X))$ is the colimit of the groups $A_T$ along the inclusion maps \cite[Section 2.7]{Haefliger92Extension}, implying $\pi_1(A(X)) = A$. It follows that the classifying space of $A(X)$ is $BA(X) = \hat\Phi \times_A EA$ \cite[Prop. 3.2.3]{Haefliger92Extension}, and thus the universal cover is
    \[
        \widetilde{BA}(X) = \hat\Phi \times EA,
    \]
    which is homotopy equivalent to $\hat \Phi$. This shows that $\hat \Phi$ is simply connected. 
\end{proof}

\subsection{The metric on \texorpdfstring{$\hat \Phi$}{Phi hat}}

In order to put a metric on $\hat \Phi$, we first note the following
\begin{lemma}
    The complex $\hat \Phi$ is 2-dimensional.
\end{lemma}
\begin{proof}
    Suppose we have a 3-simplex $[ \alpha_0 A_{T_0}, \alpha_1 A_{T_1}, \alpha_2 A_{T_2}, \alpha_3 A_{T_3} ]$ of $\hat \Phi$. By Lemma \ref{lemma:cosetinclusion}, we then have a chain $T_0 < T_1 < T_2 < T_3$. In particular, $|T_2| \geq 2$. The only sets of $\mathcal S^\ell$ with cardinality at least 2 are either $S_i$ for some $i$ or an edge $\{s_i,s_j\}$. But in either case, there is no element of $\mathcal S^\ell$ containing $T_2$, a contradiction.
\end{proof}

As a consequence of the proof of the Lemma, there are only two kinds of top-dimensional simplices of $\hat \Phi$: the first is $[\alpha_0 1, \alpha_1 A_{\{s\}}, \alpha_2 A_i]$ for a vertex $s \in S_i$ of an edge between $\Gamma_i$ and some $\Gamma_j$, and the second is $[\alpha_0 1, \alpha_1 A_{\{s_i\}}, \alpha_2 A_{\{s_i,s_j\}}]$ for $\{s_i,s_j\}$ an edge between $\Gamma_i$ and $\Gamma_j$.

We now put a metric on the two kinds of 2-simplices of $\hat \Phi$. First, consider $[\alpha_0 1, \alpha_1 A_{\{s\}}, \alpha_2 A_i]$. We give this simplex the metric of a Euclidean isosceles right triangle with right angle at $\alpha_1 A_{\{s\}}$ and whose legs have length 1. Pictorially, we have
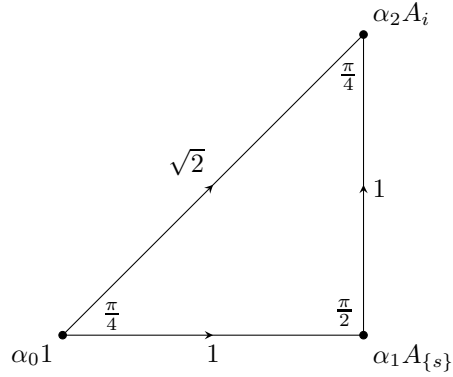
\begin{center}
\def\sidelength{4}
\begin{tikzpicture}[decoration={
    markings,
    mark=at position 0.5 with {\arrow{stealth}}}
    ]
    
    \coordinate (1) at (0,0);
    \coordinate (A1) at (\sidelength,\sidelength);
    \coordinate (As) at (\sidelength,0);

    \draw[postaction={decorate}] (1) -- (As);
    \draw[postaction={decorate}] (1) -- (A1);
    \draw[postaction={decorate}] (As) -- (A1);

    \draw[fill] (1) circle [radius=0.05];
    \node [below left] at (1) {$\alpha_0 1$};
    \node [above right] at (0.4,-0.05) {$\frac{\pi}{4}$};

    \draw[fill] (A1) circle [radius=0.05];
    \node [above right] at (A1) {$\alpha_2A_{i}$};
    \node [below left] at (\sidelength+0.05, \sidelength-0.25) {$\frac{\pi}{4}$};

    \draw[fill] (As) circle [radius=0.05];
    \node [below right] at (As) {$\alpha_1A_{\{s\}}$};
    \node [above left] at (As) {$\frac{\pi}{2}$};
    
    \node[below] at (\sidelength/2, 0) {$1$};
    \node[right] at (\sidelength,\sidelength/2-0.05) {$1$};
    \node[above left] at (\sidelength/2,\sidelength/2) {$\sqrt{2}$};
\end{tikzpicture}
\end{center}
The arrows here denote the inclusion of the relevant groups.

Now consider a simplex of the form $\Delta = [\alpha_0 1, \alpha_1 A_{\{s_i\}}, \alpha_2 A_{\{s_i,s_j\}}]$ for $e = \{s_i,s_j\}$ an edge between $\Gamma_i$ and $\Gamma_j$.

\subsubsection{Case 1: A disjoint edge}

Suppose that $e$ is disjoint from all other edges between any $\Gamma_k$ and $\Gamma_\ell$. We then put a similar metric on $\Delta$ as in the previous case; namely,
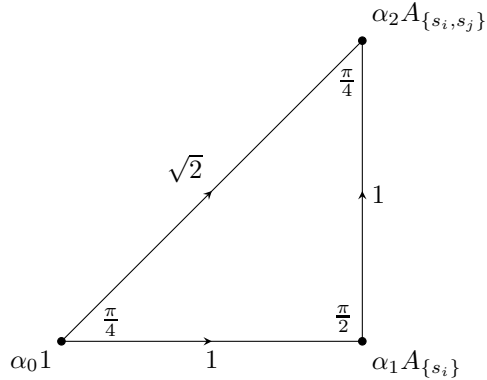
\begin{center}
\def\sidelength{4}
\begin{tikzpicture}[decoration={
    markings,
    mark=at position 0.5 with {\arrow{stealth}}}
    ]
    
    \coordinate (1) at (0,0);
    \coordinate (A1) at (\sidelength,\sidelength);
    \coordinate (As) at (\sidelength,0);

    \draw[postaction={decorate}] (1) -- (As);
    \draw[postaction={decorate}] (1) -- (A1);
    \draw[postaction={decorate}] (As) -- (A1);

    \draw[fill] (1) circle [radius=0.05];
    \node [below left] at (1) {$\alpha_0 1$};
    \node [above right] at (0.4,-0.05) {$\frac{\pi}{4}$};

    \draw[fill] (A1) circle [radius=0.05];
    \node [above right] at (A1) {$\alpha_2A_{\{s_i,s_j\}}$};
    \node [below left] at (\sidelength+0.05, \sidelength-0.25) {$\frac{\pi}{4}$};

    \draw[fill] (As) circle [radius=0.05];
    \node [below right] at (As) {$\alpha_1A_{\{s_i\}}$};
    \node [above left] at (As) {$\frac{\pi}{2}$};
    
    \node[below] at (\sidelength/2, 0) {$1$};
    \node[right] at (\sidelength,\sidelength/2-0.05) {$1$};
    \node[above left] at (\sidelength/2,\sidelength/2) {$\sqrt{2}$};
\end{tikzpicture}
\end{center}

\subsubsection{Case 2: A non-disjoint edge} Now suppose that $e$ shares a vertex with some other edge between $\Gamma_i$ and $\Gamma_j$. Then we still put the metric of a Euclidean right triangle on $\Delta$, but it will no longer be isosceles. Specifically, the metric we put on $\Delta$ still assigns a right angle to the vertex $\alpha_1 A_{\{s_i\}}$, but now places an angle of $3\pi/8$ to $\alpha_0 1$ and an angle of $\pi/8$ to $\alpha_2 A_{\{s_i,s_j\}}$. Moreover, importantly, the 1-simplex $[\alpha_01, \alpha_1 A_{\{s_i\}}]$ is given length 1. The diagram for this case is
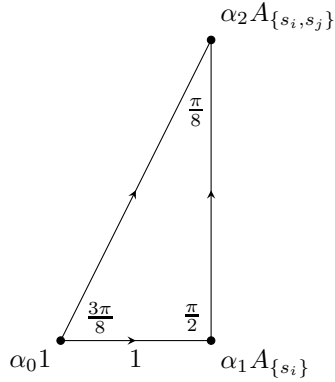
\begin{center}
\def\sidelength{3}
\begin{tikzpicture}[decoration={
    markings,
    mark=at position 0.5 with {\arrow{stealth}}}
    ]
    
    \coordinate (1) at (1,0);
    \coordinate (A1) at (\sidelength,\sidelength+1);
    \coordinate (As) at (\sidelength,0);

    \draw[postaction={decorate}] (1) -- (As);
    \draw[postaction={decorate}] (1) -- (A1);
    \draw[postaction={decorate}] (As) -- (A1);

    \draw[fill] (1) circle [radius=0.05];
    \node [below left] at (1) {$\alpha_0 1$};
    \node [above right] at (1.2,-0.025) {$\frac{3\pi}{8}$};

    \draw[fill] (A1) circle [radius=0.05];
    \node [above right] at (A1) {$\alpha_2A_{\{s_i,s_j\}}$};
    \node [below left] at (\sidelength+0.05, \sidelength+0.35) {$\frac{\pi}{8}$};

    \draw[fill] (As) circle [radius=0.05];
    \node [below right] at (As) {$\alpha_1A_{\{s_i\}}$};
    \node [above left] at (As) {$\frac{\pi}{2}$};
    
    \node[below] at (2,0) {$1$};
\end{tikzpicture}
\end{center}

In order to show this properly defines a piecewise Euclidean metric on $\hat \Phi$, we examine the gluings between adjacent simplices. We begin with a simplex of the form $\Delta = [\alpha_0 1, \alpha_1 A_{\{s\}}, \alpha_2 A_i]$. The only type of simplex $\Delta$ can be adjacent to which is not of the same type is one of the form $\Delta' = [\alpha_0 1, \alpha_1 A_{\{s\}}, \alpha_2' A_{\{s,t\}}]$ with $\{s,t\}$ an edge between $\Gamma_i$ and $\Gamma_j$, and $t$ a vertex of $\Gamma_j$. These simplices are glued only along the edge $[\alpha_01, \alpha_1 A_{\{s\}}]$, and within both simplices we have assigned this edge a length of 1.

Now consider $\Delta = [\alpha_0 1, \alpha_1 A_{\{s_i\}}, \alpha_2 A_{\{s_i,s_j\}}]$ for $\{s_i,s_j\}$ an edge between $\Gamma_i$ and $\Gamma_j$. The case where $\Delta$ is adjacent to a simplex of the form $[\alpha_0 1, \alpha_1 A_{\{s\}}, \alpha_2 A_i]$ was covered above. So, consider an adjacent simplex of the form $[\alpha_0' 1, \alpha_1 A_{\{s_i\}}, \alpha_2 A_{\{s_i,s_j\}}]$ or $[\alpha_0 1, \alpha_1' A_{\{s_i'\}}, \alpha_2 A_{\{s_i,s_j\}}]$. In either case, the metric put on the simplices is the same as that of $\Delta$ as this metric only depended on the edge $\{s_i,s_j\}$, so there is no issue with the gluing. 

It remains to check the simplices of the form $\Delta' = [\alpha_0 1, \alpha_1 A_{\{s_i\}}, \alpha_2 A_{\{s_i',s_k\}}]$ for an edge $\{s_i',s_j'\}$ and $s_k \in \Gamma_k$ for some $k \not= i$. 
By Lemma \ref{lemma:cosetinclusion}, since $\alpha_1 A_{\{s_i\}} \subseteq \alpha_2 A_{\{s_i',s_k\}}$, we have $\{s_i\} \subseteq \{s_i',s_k\}$, and since $s_k \in \Gamma_k$ we must have $s_i' = s_i$. Thus if this is to be a simplex distinct from $\Delta$, we must have $s_k \not= s_j$, so $\{s_i,s_k\}$ and $\{s_i,s_j\}$ are both edges which are not distinct. Thus the metrics on $\Delta$ and $\Delta'$ are the same, so they may be glued as required.

\section{Links}

The purpose of this section is to show the following.

\begin{prop} \label{prop:hatphicat0}
    The complex $\hat \Phi$ (with the above metric) is $\mathrm{CAT}(0)$ (hence contractible). 
\end{prop}

To do this, we compute the link at each relevant vertex of $\hat \Phi$ and show that the link condition is satisfied.
Let us briefly recall the relevant definitions. (For more details, see \cite{bridson2013metric}.)

\begin{defn}[Link of a vertex]
    Let $K$ be a polyhedral complex and $v$ a vertex of $K$. 
    Then the link of $v$ in $K$, denoted $\mathrm{lk}_{K}(v)$, is the $\varepsilon$-sphere of $K$ centered at $v$. We give the link a cell structure coming from the intersection of the sphere with the cell structure of $K$. The link is endowed with a natural spherical metric inherited from the $\varepsilon$-sphere. 
\end{defn}

In the case of the geometric realization of an abstract simplicial complex (such as $\hat \Phi$), we can give an explicit description of the link of a vertex using the underlying set.
Let $[\alpha A_T]$ be a vertex of $\hat \Phi$ (so $\alpha A_T \in A \mathcal S^\ell$). 
Then the vertex set of $\mathrm{lk}_{\hat \Phi}([\alpha A_T])$ is
\[
    \{\, \alpha' A_{T'} \,:\, \alpha' A_{T'} \subseteq \alpha A_T\,\} \cup \{\, \alpha'' A_{T''} \,:\, \alpha''A_{T''} \supseteq \alpha A_T \,\}.
\]
But by Lemma \ref{lemma:cosetinclusion}, this is the same as the set
\[
    \{\,\alpha' A_{T'} \,:\, \alpha' A_{T'} \subseteq \alpha A_T \,\} \cup \{\, \alpha A_{T''} \,:\, {T''} \supseteq T \,\}.
\]
A collection of vertices $\alpha_0 A_{T_0} < \dots < \alpha_j A_{T_j} <\alpha A_{T'_0} < \dots \alpha A_{T'_k}$ span a $(j+k)$-simplex of $\mathrm{lk}_{\hat \Phi}([\alpha A_T])$ if and only if 
\[ [ \alpha_0 A_{T_0} < \dots < \alpha_j A_{T_j} < \alpha A_T <  \alpha A_{T'_0} < \dots \alpha A_{T'_k} ]\]
is a $(j+k+1)$-simplex of $\hat \Phi$.
In the case of $\hat \Phi$, we can say slightly more than this. Our complex $\hat\Phi$ is 2-dimensional, so the link of any vertex is 1-dimensional. Moreover, the link of a simplicial complex is itself a simplicial complex, so the link here is always a simplicial graph. 

We can also explicitly describe the spherical metric on each link in $\hat \Phi$. If $[\alpha A_T]$ is a vertex of $\hat \Phi$ and $e = [\alpha_0 A_{T_0}, \alpha_1 A_{T_1}]$ is an edge of $\mathrm{lk}_{\hat\Phi}([\alpha A_T])$, then the length of $e$ is the angle assigned above to the vertex corresponding to $\alpha A_T$ in the simplex of $\hat \Phi$ spanned by the vertices $\alpha A_T$, $\alpha_0 A_{T_0}$, $\alpha_1 A_{T_1}$.

\begin{defn}
    We say that a polyhedral complex $K$ satisfies the \emph{link condition} if for each vertex $v$ of $K$, the link $\mathrm{lk}_K(v)$ is a $\mathrm{CAT}(1)$ space (under the induced spherical metric).
\end{defn}

To show $\hat \Phi$ is $\mathrm{CAT}(0)$, we make use of the following criterion, proven in \cite{bridson2013metric}.

\begin{lemma} \label{thm:catcrit}
    If $K$ is a Euclidean polyhedral complex (meaning each cell of $K$ has the metric of a Euclidean polytope) and $K$ is simply connected, then $K$ is $\mathrm{CAT}(0)$ if and only if it satisfies the link condition.
\end{lemma}

Since our complex $\hat \Phi$ is 2-dimensional, to verify our links are $\mathrm{CAT}(1)$, we can use the following equivalent condition, also proven in \cite{bridson2013metric}.

\begin{lemma} \label{lemma:twodlink}
    A 2-dimensional Euclidean simplicial complex $K$ satisfies the link condition if and only if for each vertex $v$ of $K$, every embedded closed loop in $\mathrm{lk}_K(v)$ has length at least $2\pi$.
\end{lemma}

We now turn to examining the links of our complex in detail. Since each vertex of $\hat \Phi$ is a translate of one of the cosets $A_T$, it suffices to just compute the link at $A_T$ for $T \in \mathcal S^\ell$.

\subsection{Case 1: \texorpdfstring{$T = S_i$}{T = S\_i}}
Let us first examine the link of $A_i$ for fixed $i$. 
The vertex set of this link can be decomposed as
\begin{align*}
    \{\, \alpha1 \,:\, \alpha \in A_i \,\} \qquad \text{ and } \qquad \{\,\alpha A_{\{s\}} \,:\, \alpha \in A_i, s \in S_i \,\}.
\end{align*}
It is easily seen that there is no edge between any two vertices which are in the same set, meaning the link is a bipartite graph. 
By definition, we can only have an edge when $\alpha1 \subseteq \alpha'A_{\{s\}}$, or in other words, when $\alpha \in \alpha'A_{\{s\}}$.

To show that the shortest embedded closed loop in $A_i$ has length at least $2\pi$, we claim that any embedded closed loop in $\mathrm{lk}_{\hat \Phi}(A_i)$ must have at least 8 edges. 
Since the link is a bipartite graph, we know the edge length of any cycle is even and at least 4. So, we only need to verify that there are no cycles of edge length 4 or 6. 

Suppose we have a loop with 4 edges. Then by our discussion regarding the possible edges in the link, this loop must have the form
{\setlength\mathsurround{0pt}
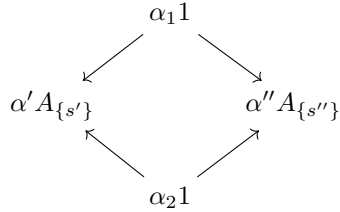
\begin{center}
     \begin{tikzcd}[column sep=small]
             &  \alpha_1 1 \dlar \arrow[dr]   &               \\
        \alpha' A_{\{s'\}}  &                            & \alpha'' A_{\{s''\}} \\
             & \alpha_2 1 \arrow[ul] \arrow[ur]   &
    \end{tikzcd}
\end{center} 
}
(The arrows correspond to inclusions; the paths we consider are not directed.)
This gives us equations of the form
\begin{align*}
\alpha'(s')^{k_1} &= \alpha_1 = \alpha''(s'')^{j_1} \\
\alpha'(s')^{k_2} &= \alpha_2 = \alpha''(s'')^{j_2}.
\end{align*}
Or rewriting, 
\begin{align*}
    (s'')^{j_1}(s')^{-k_1} &= ( \alpha'')^{-1}\alpha' = (s'')^{j_2}(s')^{-k_2},
\end{align*}
implying
\begin{align*}
    (s'')^{j_1-j_2} &= (s')^{k_1-k_2}.
\end{align*}
Since we're assuming the loop is embedded, $s' \not= s''$ (otherwise, two cosets of the same subgroup $A_{s'} = A_{s''}$ would intersect non-trivially, and thus be the same), and $\alpha_1 \not= \alpha_2$, so $k_1\not=k_2$ and $j_1\not= j_2$. However, these are distinct generators, so this cannot happen by Lemma \ref{lemma:prodgen}. Thus this loop is not embedded. 

Now suppose we have a loop with 6 edges. This loop has the form
{\setlength\mathsurround{0pt}
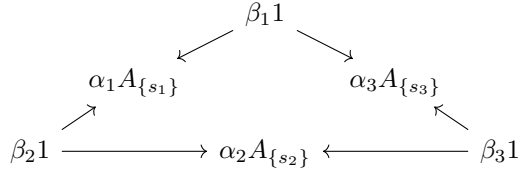
\begin{center}
     \begin{tikzcd}[column sep=tiny, row sep=small]
          &   &  \beta_1 1 \arrow[dl] \arrow[dr]   &               \\
        & \alpha_1 A_{\{s_1\}}  &                            & \alpha_3 A_{\{s_3\}} \\
        \beta_2 1 \arrow[ur] \arrow[rr]   &  & \alpha_2 A_{\{s_2\}}  & & \beta_3 1 \arrow[ul] \arrow[ll]
     \end{tikzcd}
\end{center}
}
Since the loop is embedded, each $\beta_i$ is distinct and at most one of the $\beta_i$ can be the identity, so assume $\beta_1 \not= 1$ and $\beta_2 \not= 1$.
Then since $\beta_1 \not= 1$, we must have that $s_1 \not = s_3$ (as before, if we did have $s_1 = s_3$, then the cosets $\alpha_1 A_{\{s_1\}}$ and $\alpha_3 A_{\{s_3\}}$ would be cosets of the same subgroup $A_{\{s_1\}} = A_{\{s_3\}}$ which intersect non-trivially, thus would be the same coset). Similarly, $s_1 \not= s_2$. 

From our diagram, we see that we have equations
\begin{align*}
     \alpha_1 s_1^{k_1}     &= \beta_1 = \alpha_3 s_3^{k_3}    \\
     \alpha_2 s_2^{j_2}     &= \beta_2 = \alpha_1 s_1^{j_1}    \\
     \alpha_3 s_3^{\ell_3}  &= \beta_3 = \alpha_2 s_2^{\ell_2}  .
\end{align*}
for some $k_i,j_i,\ell_i \in \mathbb{Z}$. 
Then we see that
\begin{align*}
    s_1^{j_1 - k_1} 
        &= s_1^{-k_1}s_1^{j_1} \\
        &= (\alpha_1^{-1}\beta_1)^{-1}(\alpha_1^{-1}\beta_2) \\
        &= \beta_1^{-1}\beta_2,
\end{align*}
and similarly,
\[
    s_2^{\ell_2 - j_2} = \beta_2^{-1} \beta_3, \qquad s_3^{k_3 - \ell_3} = \beta_3^{-1} \beta_1
\]
Note that since the $\beta_i$ are distinct, none of these exponents are zero.
But,
\begin{align*}
    s_1^{j_1 - k_1} &= \beta_1^{-1}\beta_2\\
        &= \beta_1^{-1}(\beta_3\beta_3^{-1})\beta_2 \\
        &= (\beta_3^{-1}\beta_1)^{-1}(\beta_2^{-1}\beta_3)^{-1} \\
        &= s_3^{\ell_3 - k_3}  s_2^{j_2 - \ell_2}  .
\end{align*}
This means $s_1^{j_1 - k_1} \in A_{\{s_2,s_3\}}$, and so $A_{\{s_1\}} \cap A_{\{s_2,s_3\}} \not= 1$ since $j_1 - k_1 \not= 0$. But then by \cite{vdl1983}, this would mean $\{s_1\} \cap \{s_2, s_3\} \not= \varnothing$, a contradiction. 
Thus this loop cannot be embedded.

Therefore, each embedded loop in $\mathrm{lk}_{\hat\Phi}(A_i)$ has at least 8 edges.
The spherical metric on the link assigns each of these edges a length of $\pi/4$, so the shortest possible length of an embedded loop is $2\pi$.

\subsection{Case 2: \texorpdfstring{$T = \{s\}$}{T = \{s\}}}

Now we look at the link of $A_{\{s\}}$ with $s \in \mathrm{Vert}(\Gamma_i)$ a vertex of an edge between $\Gamma_i$ and $\Gamma_j$. In this case the link is again a bipartite graph: 
the vertices can be divided into the sets
\begin{align*}
    &\{\, \alpha 1 \,:\, \alpha \in A_{\{s\}} \,\}, \qquad \text{and} \\ & \{A_i\} \cup \{\, A_{\{s, s_k\}} \,:\, \{s, s_k\} \text{ is an edge between $\Gamma_i$ and $\Gamma_k$, $k \not= i$} \,\}.
\end{align*}
So, every embedded loop has at least 4 edges.
The spherical metric on the link assigns a length of $\pi/2$ to each of these edges, implying the length of every embedded loop is at least $2\pi$.

\subsection{Case 3: \texorpdfstring{$T = \{s_i,s_j\}$}{T = \{s\_i,s\_j\}}}

The link of $A_T$ for $T = \{s_i,s_j\}$, $s_i \in S_i$, $i \not= j$, is slightly different, as there are two cases to consider. However, in both cases the minimal number of edges in an embedded loop are the same.

\begin{lemma}
    If $T = \{s_i,s_j\}$ is an edge between $\Gamma_i$ and $\Gamma_j$ for $i\not=j$, then each embedded loop in $\mathrm{lk}_{\hat\Phi}([A_T])$ has at least $4m(s_i,s_j)$ edges.
\end{lemma}
\begin{proof}
    The link of $A_T$ has vertex set which can be split into
    \begin{align*}
        \{\, \alpha 1 \,:\, \alpha \in A_T \,\} \qquad \text{ and } \qquad \{\,\alpha A_{s_k} \,:\, \alpha \in A_T,\, k = i,j \,\},
    \end{align*}
    on which the link is a bipartite graph. By applying the natural $A_T$ action on the link, we may consider only loops which contain the vertex $1$. Namely, we may consider only loops of the form
    
    {\setlength\mathsurround{0pt}
    \begin{center}
        \begin{tikzcd}[row sep=small]
            \alpha_1 A_{t_1} & \beta_1 \arrow{l} \arrow{r} & \alpha_2 A_{t_2} & \beta_2 \arrow{l} \arrow{d}\\
            1 \arrow{u} \arrow{d} & & & \vdots \\
            \alpha_{n} A_{t_{n}} &\beta_{n-1} \arrow{l} \arrow{r} & \alpha_n A_{t_{n-1}} & \beta_{n-2}  \arrow{l} \arrow{u}
        \end{tikzcd}
    \end{center}
    }
    where each $\alpha_k \in A_T$ and each $t_k \in T$. This loop has $2n$ edges. Moreover, assuming this loop is embedded, this gives rise to a (reduced) word in $s_i$ and $s_j$ of syllable length at least $n$ (see \cite[Section 4]{appel1983artin} for the definition of syllable length) which is equal to the identity in $A_T$. By \cite[Lemma 6]{appel1983artin}, this means $n \geq 2m(s_i,s_j)$. Thus, this loop has at least $4m(s_i,s_j)$ edges.
\end{proof}

Now, we can compute the length of these loops in each given link.

\subsubsection{Case 3a: A disjoint edge}

If $\{s_i,s_j\}$ is disjoint from every other edge between the subgraphs in $\{\Gamma_k\}$, then the spherical metric on the link of $A_T$ implies that the length of each edge here is $\pi/4$. So, the length of any embedded loop is at least $(4m(s_1,s_2))(\pi/4) = \pi m(s_1,s_2)$. Since $m(s_1,s_2) \geq 2$, this loop has length at least $2\pi$, as required.

\subsubsection{Case 3b: A non-disjoint edge}

If this edge is not disjoint from every other edge between the subgraphs in $\{\Gamma_k\}$, the metric we have assigned implies that the length of each edge is $\pi/8$. So, the length of any embedded loop is at least $(4m(s_1,s_2))(\pi/8) = \pi m(s_1,s_2)/2$. But in this case, we have also assumed $m(s_1,s_2) \geq 4$, so the length of this loop is still at least $2\pi$.

\subsection{Case 4: \texorpdfstring{$T = \varnothing$}{T = empty set}}

It remains to check the link of the trivial coset $1$. Note again that this link is bipartite, with a partition of the vertices given by
\begin{align*}
    &\{\, A_{\{s\}} \,:\, s \text{ a vertex of an edge between the subgraphs in } \{\Gamma_k\} \,\}, \qquad \text{ and} \\
    & \{\, A_i \,:\, \text{ each } i\,\} \cup  \{\, A_{\{s_i, s_j\}} \,:\, \{s_i, s_j\} \text{ an edge between $\Gamma_i$ and $\Gamma_j$, $k \not= i$} \,\}.
\end{align*}

We first verify that there are no embedded loops with 4 edges. Suppose we had such an embedded loop, say
{\setlength\mathsurround{0pt}
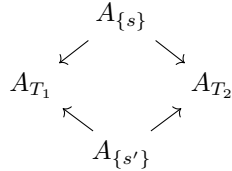
\begin{center}
     \begin{tikzcd}[column sep=tiny, row sep=small]
             &  A_{\{s\}} \arrow[dl] \arrow[dr]   &               \\
        A_{T_1}  &                            & A_{T_2} \\
             & A_{\{s'\}} \arrow[ul] \arrow[ur]   &
    \end{tikzcd}
\end{center} 
}
Since this loop is embedded, $s \not= s'$. Thus by Lemma \ref{lemma:cosetinclusion}, both $T_1$ and $T_2$ contain $\{s,s'\}$. If $s$ and $s'$ are in the same vertex set $S_i$, then we must have $T_1 = T_2 = S_i$ by our definition of $\mathcal S^\ell$. Similarly, if they are in distinct vertex sets, then both $T_1$ and $T_2$ must exactly be the edge $\{s,s'\}$. In either case, we have a contradiction.

It is entirely possible that we have embedded loops of length 6. Suppose
{\setlength\mathsurround{0pt}
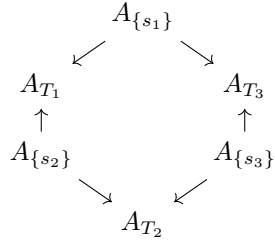
\begin{center}
    \begin{tikzcd}[column sep=tiny, row sep=small]
                          & A_{\{s_1\}} \arrow[dl] \arrow[dr] & \\
        A_{T_1}  &               &  A_{T_3} \\
        A_{\{s_2\}} \arrow[u] \arrow[dr]     &               &  A_{\{s_3\}} \arrow[u] \arrow[dl] \\
                          & A_{T_2} & \\
    \end{tikzcd}
\end{center}
}
is such a loop. If each pair $\{s_i,s_j\}$ is an edge between the family of subgraphs $\{\Gamma_i\}$, 
then the $T_i$ must be the edges
\begin{align*}
    T_1 &= \{s_1,s_2\}, \\
    T_2 &= \{s_2,s_3\}, \\
    T_3 &= \{s_3,s_1\}
\end{align*}
since these are the only sets of $\mathcal S^\ell$ which satisfy the containments implied by the diagram.
But none of these edges are disjoint, so the metric we've put on $\hat \Phi$ assigns the following edge lengths to this path:
{\setlength\mathsurround{0pt}
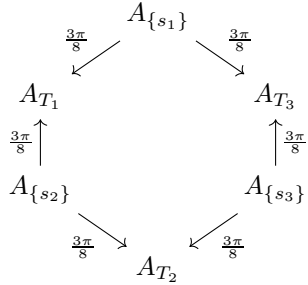
\begin{center}
    \begin{tikzcd}[column sep=small]
                          & A_{\{s_1\}} \arrow{dl}[swap]{\frac{3\pi}{8}} \arrow{dr}{\frac{3\pi}{8}} & \\[-0.5em]
        A_{T_1}  &               &  A_{T_3} \\
        A_{\{s_2\}}  \arrow{u}{\frac{3\pi}{8}} \arrow{dr}[swap]{\frac{3\pi}{8}}     &               &  A_{\{s_3\}} \arrow{u}[swap]{\frac{3\pi}{8}} \arrow{dl}{\frac{3\pi}{8}} \\[-0.5em]
                          & A_{T_2} & \\
    \end{tikzcd}
\end{center}
}
And thus this loop has length at least $2\pi$. Now suppose two of the vertices $s_i$ are in the same vertex set and the other is in a distinct vertex set. Without loss of generality, we can take $s_1,s_2 \in S_i$ and $s_3 \in S_j$ with $i \not= j$. Then the only set $T_1 \in \mathcal S^\ell$ containing both $s_1$ and $s_2$ is $T_1 = S_i$, and so we now have 
\begin{align*}
    T_1 &= S_i, \\
    T_2 &= \{s_2,s_3\}, \\
    T_3 &= \{s_3,s_1\}.
\end{align*}
The metric on $\hat \Phi$ then assigns the following edge lengths:
{\setlength\mathsurround{0pt}
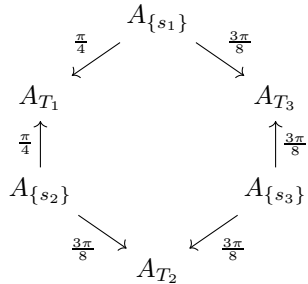
\begin{center}
    \begin{tikzcd}[column sep=small]
                          & A_{\{s_1\}} \arrow{dl}[swap]{\frac{\pi}{4}} \arrow{dr}{\frac{3\pi}{8}} & \\[-0.5em]
        A_{T_1}  &               &  A_{T_3} \\
        A_{\{s_2\}}  \arrow{u}{\frac{\pi}{4}} \arrow{dr}[swap]{\frac{3\pi}{8}}     &               &  A_{\{s_3\}} \arrow{u}[swap]{\frac{3\pi}{8}} \arrow{dl}{\frac{3\pi}{8}} \\[-0.5em]
                          & A_{T_2} &
    \end{tikzcd}
\end{center}
}
which is still at least $2\pi$.
We note that it is not possible to have $s_1,s_2,s_3 \in S_i$ for any $i$, since then $T_1 = T_2 = T_3 = S_i$, and this loop would not be embedded.

Finally, if we have a loop with 8 edges in this link, then the length of each edge under our metric is at least $\pi/4$, and thus the length of this loop would be at least $2\pi$ as well.

This concludes every possibility for $T$, so it follows that $\hat \Phi$ satisfies the link condition by Lemma \ref{lemma:twodlink}. By Lemma \ref{lemma:simcon}, $\hat \Phi$ is simply connected, so by Lemma \ref{thm:catcrit}, $\hat \Phi$ is $\mathrm{CAT}(0)$, and thus contractible, as desired.

\section{The \texorpdfstring{$K(\pi,1)$}{K(pi,1)} Conjecture}

In this section only, \textbf{we assume that each $A_{\Gamma_i}$ satisfies the $K(\pi,1)$ conjecture}. In addition, we assume that $A_\Gamma$ is not spherical-type. (Since the $K(\pi,1)$ conjecture is known for spherical-type Artin groups, there is no loss of generality in making this assumption.)  

We will use the following by Paris and Godelle in \cite{godelle2012k}.

\begin{defn}
    Let $(A,S)$ be an Artin-Tits system and let $\mathcal S$ be a family of subsets of $S$. Then $\mathcal S$ is \emph{complete and $K(\pi,1)$} if the following are satisfied:
    \begin{enumerate}
        \item If $T \in \mathcal S$ and $T' \subseteq T$, then $T' \in \mathcal S$,
        \item $(A_T, T)$ satisfies the $K(\pi,1)$ conjecture for each $T \in \mathcal S$, and
        \item If $A_T$ is spherical-type, then $T \in \mathcal S$.
    \end{enumerate}
    Then let
    \[ A\mathcal S = \{\, \alpha A_T \,:\, \alpha \in A,\, T \in \mathcal S \,\}, \]
    and let $\Phi(A, \mathcal S)$ denote the geometric realization of the derived complex of $A \mathcal S$.
\end{defn}

The relevant result for us is 

\begin{thm} \label{thm:complhomtype}
    \textbf{\cite[Theorem 3.1]{godelle2012k}} \:
    Let $(A,S)$ be an Artin-Tits system and let $\mathcal S$ be a complete and $K(\pi,1)$ family of subsets of $S$. Then $\Phi(A,\mathcal S)$ has the same homotopy type as the universal cover of $N(W)$.
\end{thm}

Our family $\mathcal S^\ell$ is not itself complete and $K(\pi,1)$, so we cannot directly apply this result. Instead, we show that $\hat \Phi$ is homotopic to $\Phi \coloneqq \Phi(A, \overline {\mathcal S})$ for a certain complete and $K(\pi,1)$ collection $\overline {\mathcal S}$ which we define as follows:
the sets of $\overline {\mathcal S}$ are the subsets of $S$ consisting of (1) the sets in $\mathcal S^\ell$, and (2) every subset of $S_i$. 

\begin{lemma} \label{lemma:barScomplete}
    $\overline{\mathcal S}$ is a complete and $K(\pi,1)$ family of subsets of $S$.
\end{lemma}
\begin{proof}
    First we note that (1) and (2) are satisfied immediately by our definition of $\overline{\mathcal S}$ (to see (2), note that a standard parabolic subgroup satisfies the $K(\pi,1)$ conjecture whenever the original group does by \cite[Cor. 2.4]{godelle2012k}). It remains to show that $\overline{\mathcal S}$ contains all spherical-type generating sets. 
    
    Suppose $\Gamma'$ is a full subgraph of $\Gamma$ such that $A_{\Gamma'}$ is spherical-type, and let $T = \mathrm{Vert}(\Gamma')$. If $T \subseteq S_i$, then we already have $T \in  \overline{\mathcal S}$. So, suppose there are $t_1,t_2\in T$ with $t_1$ and $t_2$ in distinct vertex sets, say $t_1 \in S_i$ and $t_2 \in S_j$ for $i \not= j$.
    
    If $T = \{t_1,t_2\}$ then since we're assuming $A_{\Gamma'}$ is spherical-type, we must have that $\{t_1,t_2\}$ is an edge of $\Gamma$, and thus $T \in \overline{\mathcal S}$. In other words, whenever $|T| = 2$ and $T \not\subset S_k$ for any $k$, we must have that $T$ is an edge of $\Gamma$, so $T \in \overline{\mathcal S}$.
    
    Suppose $|T| > 2$, and let $t_3 \in T$ be distinct from $t_1$ and $t_2$. If any of $\{t_1,t_2\}$, $\{t_2,t_3\}$, or $\{t_3,t_1\}$ is not an edge of $\Gamma$, then $A_{\Gamma'}$ would not be spherical, so we must have that each of these are edges. 
    There are three cases to consider: either $t_3$ is in $S_1$, is in $S_2$, or is in neither. By symmetry we may consider only the cases where $t_3 \in S_1$ and $t_3$ is in neither. In both of these cases,
    $\{t_1,t_2\}$ and $\{t_3,t_2\}$ are distinct non-disjoint edges between the family $\{\Gamma_i\}$, so by the (REL$'$) condition, both of their labels must be at least 4. By the classification of finite Coxeter groups, we would then have that $A_{\Gamma'}$ is not spherical-type. Thus if $T \not\subset S_i$ we cannot have $|T| > 2$. 
\end{proof}

Thus, we have that $\Phi$ is homotopy equivalent to the universal cover of $N(W)$. It remains to show that

\begin{thm} \label{thm:phihatdeform}
    There is a deformation retract from $\Phi$ to $\hat \Phi$.
\end{thm}
\begin{proof}
    Note that there is a natural embedding of $\hat \Phi$ into $\Phi$ induced by the inclusion of $\mathcal S^\ell$ in $\overline{\mathcal S}$.
    
    We establish the deformation retract directly by describing the maps on each simplex.
    Let $\Delta$ be a maximal simplex of $\Phi$ (ie, one which is not a face of any other simplex). There are two types of simplices to consider. The first is 
    \[
        \Delta = [\alpha_0 1, \alpha_0 A_{\{s\}}, \alpha_1 A_{\{s,t\}}]
    \]
    for an edge $\{s,t\}$ between the family $\{\Gamma_i\}$. This is already a maximal simplex of $\hat \Phi$, so we leave it unchanged.
    In the other case, we have that \[\Delta = [\alpha_0 A_{T_0}, \alpha_1 A_{T_1}, \dots, \alpha_{n-1}A_{T_{n-1}}, \alpha_n A_{S_i}]\] for some $S_i$. Since $\Delta$ is maximal, we must have $T_0 = \varnothing$ and $T_1 = \{s\}$ for some $s \in S_i$. There are two subcases to consider.
    If $s$ is a vertex of an edge between $\Gamma_i$ and some $\Gamma_j$, then there is a natural deformation retract from $\Delta$ to the simplex $[\alpha_0 A_{T_0}, \alpha_1 A_{T_1}, \alpha_n A_{S_i}]$ of $\hat\Phi$. Otherwise, there is a natural deformation retract from $\Delta$ to the simplex $[\alpha_0 A_{T_0}, \alpha_n A_{S_i}]$ of $\hat \Phi$. Moreover, these can easily be parameterized so that we can glue deformation retracts of adjacent maximal simplices to attain a deformation retract on the entire complex $\Phi$.
\end{proof}

We have therefore proven

\relkpone*

\begin{proof}
    First, suppose $A_\Gamma$ satisfies the $K(\pi,1)$ conjecture. Then, by \cite[Cor. 2.4]{godelle2012k}, each $A_{\Gamma_i}$ also does.

    Now suppose each $A_{\Gamma_i}$ satisfies the $K(\pi,1)$ conjecture. Combining Theorem \ref{thm:complhomtype} and Lemma \ref{lemma:barScomplete}, we have that $\Phi$ is homotopy equivalent to the universal cover of $N(W)$, and by Theorem \ref{thm:phihatdeform}, $\hat\Phi$ is homotopy equivalent to $\Phi$. Thus, by Theorem \ref{prop:hatphicat0}, the universal cover of $N(W)$ is contractible.
\end{proof}

\section{Acylindrical hyperbolicity}

We conclude by showing the following

\acynhyp*

For the full definition of acylindrical hyperbolicity, we refer the reader to \cite{Bowditch2008}.

First, if there are no edges between the family $\{\Gamma_i\}$, then $A_\Gamma$ is a free product of the $A_{\Gamma_i}$, and thus is acylindrically hyperbolic by considering the action of $A_\Gamma$ on its Bass-Serre tree. So, we assume there is an edge between the family $\{\Gamma_i\}$, say $e = \{s_i,s_j\}$ with $s_i \in S_i$, $s_j \in S_j$ and $i \not= j$. By our assumptions on $\Gamma$, we may take $e$ to have label at least 3.
In this case, we make use of the following adaptation of a theorem from \cite{martin2017}. 

\begin{thm} \label{thm:hypcrit}
    \textbf{\emph{\cite[Theorem B]{martin2017}}} \:
    Let $X$ be a $\mathrm{CAT}(0)$ simplicial complex and $G$ a group acting on $X$ by simplicial isomorphisms. Suppose there is a vertex $v$ of $X$ with stabilizer $G_v$ satisfying
    \begin{enumerate}
        \item The orbits of $G_v$ on the link $\mathrm{lk}_X(v)$ are unbounded in the associated spherical metric, and
        \item $G_v$ is weakly malnormal in $G$ (ie, there exists an element $g \in G$ such that $G_v \cap gG_vg^{-1}$ is finite).
    \end{enumerate}
    Then G is either virtually cyclic or acylindrically hyperbolic.
\end{thm}

\begin{rem}
    This is a strictly weaker statement than the one given in \cite{martin2017}. The original statement of the theorem allows $X$ to be a polyhedral complex satisfying the ``Strong Concatenation Property''. By Example 2.9 and Lemma 2.11 in \cite{martin2017}, $\mathrm{CAT}(0)$ simplicial complexes always satisfy this property.
\end{rem}

We use the action of $A_\Gamma$ on our Deligne-like complex $\hat\Phi$, which we have shown is $\mathrm{CAT}(0)$ for any Artin group satisfying (REL'). We claim that $v_e \coloneqq [A_{\{s_i,s_j\}}]$ is a vertex of $\hat\Phi$ which satisfies the conditions of Theorem \ref{thm:hypcrit}.

Since we have assumed $|\mathrm{Vert}(\Gamma)| \geq 3$, there is at least one vertex $s$ of $\Gamma$ distinct from $s_i$ and $s_j$. If there are no such $s$ so that either $\{s,s_i\}$ or $\{s,s_j\}$ are edges of $\Gamma$, then $A_\Gamma$ is a free product and thus acylindrically hyperbolic by our previous remarks. In the other case, take $s$ so that one of $\{s,s_i\}$ or $\{s,s_j\}$ is an edge of $\Gamma$, and define $\Delta = \{s,s_i,s_j\}$. Then the full subgraph of $\Gamma$ on vertices $\Delta$ is connected, and,
by the (REL$'$) condition, $A_\Delta$ is a 2-dimensional Artin group. Moreover, we have assumed $m(s_i,s_j) > 2$, so $A_\Delta$ is not a right-angled Artin group. Thus, we may use the following

\begin{thm}
    \textbf{\emph{\cite[Lemma 5.7]{vaskou2021acylindrical}}} \:
    Let $A_{\Lambda}$ be a $2$-dimensional Artin group of rank at least $3$, and suppose that $\Lambda$ is connected and that $A_{\Lambda}$ is not a right angled Artin group. Then there exists an Artin subgroup $A_{\{a,b\}}$ with coefficient $3 \leq m(a,b) < \infty$ and an element $g \in A_{\Lambda}$ such that $A_{\{a,b\}} \cap g A_{\{a,b\}} g^{-1} = \{1\}$.
\end{thm}

Applying this to $A_\Delta$, we have $a,b \in \Delta$ and $g \in A_\Delta$ so that $A_{\{a,b\}} \cap g A_{\{a,b\}} g^{-1} = \{1\}$. The proof of the Theorem implies that we may take $\{a,b\} = \{s_i,s_j\}$. This shows that $v_e$ satisfies (2). To show $v_e$ satisfies (1), we use

\begin{thm}
    \textbf{\emph{\cite[Lemma 4.5]{vaskou2021acylindrical}}} \:
    Consider an Artin group $A_{\{a,b\}}$ with coefficient $3 \leq m(a,b) \leq \infty$. Then
    \[\{\, \ell_{\mathcal{S}}(g) \,:\, g \in A_{\{a,b\}} \,\} \] 
    is unbounded (where $\ell_{\mathcal S}(g)$ is the syllable length of $g$).
\end{thm}
By the same analysis in the case of loops, if $\{a,b\}$ is an edge between the family $\{\Gamma_i\}$, then reduced words in $a$ and $b$ correspond to paths in $\mathrm{lk}_{\hat\Phi}([A_{\{a,b\}}])$, and vice versa. The edge length of such a path is at least the syllable length of the given word. So, since $m(s_i,s_j) \geq 3$, we have that the action of $A_{\{s_i,s_j\}}$ on $\mathrm{lk}_{\hat\Phi}([A_{\{s_i,s_j\}}]) = \mathrm{lk}_{\hat\Phi}(v_e)$ is unbounded. Therefore, $v_e$ also satisfies (1), and thus $A_\Gamma$ is acylindrically hyperbolic.

\bibliographystyle{alpha}
\bibliography{relativelyextralarge}

\end{document}